\documentclass[a4paper,twoside,10pt]{article}

\newcommand{\mytitle}{Quantitative reconstructions in multi-modal photoacoustic and optical coherence tomography imaging}
\title{\mytitle}

\author{Peter Elbau$^1$\\{\footnotesize\href{mailto:peter.elbau@univie.ac.at}{peter.elbau@univie.ac.at}}
\and Leonidas Mindrinos$^1$\\{\footnotesize\href{mailto:leonidas.mindrinos@univie.ac.at}{leonidas.mindrinos@univie.ac.at}}
\and Otmar Scherzer$^{1,2}$\\{\footnotesize\href{mailto:otmar.scherzer@univie.ac.at}{otmar.scherzer@univie.ac.at}}}

\usepackage[utf8]{inputenc}

\usepackage{microtype}

\usepackage{amsmath}

\usepackage{longtable}

\usepackage{amssymb}

\usepackage{mathrsfs}

\usepackage{mathtools}

\usepackage{siunitx}

\usepackage[a4paper,centering,bindingoffset=0cm,marginpar=0cm]{geometry}

\usepackage{titlesec}
\usepackage[font=footnotesize,format=plain,labelfont=sc,textfont=sl,width=0.75\textwidth,labelsep=period]{caption}
\titleformat{\section}{\filcenter\sc\large}{\thesection.\;}{0em}{}
\titleformat{\subsection}[runin]{\bf}{\thesubsection.\;}{0em}{}[.]

\newpagestyle{headers}%
{%
	\headrule
	\sethead%
		[\thepage]%
		[\footnotesize\sc P.~Elbau, L.~Mindrinos, and O.~Scherzer]%
		[]%
		{}%
		{\footnotesize\sc{Quantitative Reconstructions in Multi-modal PAT/OCT Imaging}}%
		{\thepage}
}
\usepackage[backend=bibtex,maxnames=15]{biblatex}
\usepackage[pdftex,colorlinks=true,linkcolor=blue,citecolor=green,urlcolor=blue,bookmarks=true,bookmarksnumbered=true]{hyperref}
\hypersetup
{
    pdfauthor={Peter Elbau, Leonidas Mindrinos, Otmar Scherzer},
    pdfsubject={Quantitative Photoacoustic Tomography and Optical Coherence Tomography},
    pdftitle={\mytitle},
    pdfkeywords={Quantitative Photoacoustic Tomography, Optical Coherence Tomography, Hybrid imaging}
}

\usepackage[hyperref,amsmath,thmmarks]{ntheorem}
\usepackage{aliascnt}
\theoremstyle{break}
\theoremseparator{.}
\theoremheaderfont{\kern-1em\normalfont\bfseries}
\newtheorem{lemma}{Lemma}[section]

\newaliascnt{proposition}{lemma}
\newtheorem{proposition}[proposition]{Proposition}
\aliascntresetthe{proposition}

\newaliascnt{corollary}{lemma}

\aliascntresetthe{corollary}

\newaliascnt{assumptions}{lemma}

\aliascntresetthe{assumptions}

\newaliascnt{invpro}{lemma}

\aliascntresetthe{invpro}

\theorembodyfont{\normalfont}
\newaliascnt{definition}{lemma}
\newtheorem{definition}[definition]{Definition}
\aliascntresetthe{definition}

\theorembodyfont{\normalfont}
\newaliascnt{example}{lemma}

\aliascntresetthe{example}

\theorembodyfont{\normalfont}
\newaliascnt{convention}{lemma}

\aliascntresetthe{convention}

\theorembodyfont{\normalfont\itshape}
\theoremseparator{:}
\theoremheaderfont{\kern-1em\normalfont\itshape}
\newaliascnt{remark}{lemma}
\newtheorem{remark}[remark]{Remark}
\aliascntresetthe{remark}

\theoremstyle{nonumberplain}
\theorembodyfont{\normalfont}
\newtheorem{proof}{Proof}

    \usepackage[ntheorem]{empheq}

\usepackage{bbm}
\usepackage{bm}
\newcommand{\R}{\mathbbm{R}}

\newcommand{\C}{\mathbbm{C}}
\newcommand{\N}{\mathbbm{N}}
\renewcommand{\b}{\bm}

\newcommand{\id}{\mathbbm{1}}

\newcommand{\e}{\mathrm e}
\renewcommand{\i}{\mathrm i}
\renewcommand{\d}{\,\mathrm d}

\let\RE\Re
\let\Re=\undefined
\DeclareMathOperator{\Re}{\RE e}
\let\IM\Im
\let\Im=\undefined
\DeclareMathOperator{\Im}{\IM m}
\DeclareMathOperator{\sign}{sign}
\DeclareMathOperator{\curl}{curl}
\let\div=\undefined
\DeclareMathOperator{\div}{div}
\DeclareMathOperator{\grad}{grad}
\DeclareMathOperator{\supp}{supp}

\begin{document}

\maketitle
\hspace*{1em}
\parbox[t]{0.49\textwidth}{\footnotesize
\hspace*{-1ex}$^1$Computational Science Center\\
University of Vienna\\
Oskar-Morgenstern-Platz 1\\
A-1090 Vienna, Austria}
\parbox[t]{0.4\textwidth}{\footnotesize
\hspace*{-1ex}$^2$Johann Radon Institute for Computational\\
\hspace*{0.1em}and Applied Mathematics (RICAM)\\
Altenbergerstra{\ss}e 69\\
A-4040 Linz, Austria}

\vspace*{2em}

\begin{abstract}
In this paper we perform quantitative reconstruction of the electric susceptibility and the Gr\"uneisen parameter  of a non-magnetic linear dielectric medium using measurement of a multi-modal photoacoustic and optical  coherence tomography system. We consider the mathematical model presented in \cite{ElbMinSch17}, where a Fredholm integral equation of the first kind for the Gr\"uneisen parameter was derived. For the numerical solution of the integral equation we consider a Galerkin type method.
\end{abstract}

\section{Introduction}\label{section_int}

Tomographic imaging techniques visualize the inner structure of probes. Particularly relevant for this work are Optical Coherence Tomography (OCT) and Photoacoustic (PAT). In OCT a sample is placed in an interferometer and is illuminated by light pulses. Then, the backscattered light is measured far from the medium, see for instance \cite{Bre06, DreFuj15, Fer96}. 
PAT visualizes the capability of a medium to transform optical (infrared) waves into ultrasound waves to be measured on the surface of the medium \cite{HalSchBurPal04, Wan08, XuWan06}. PAT is called coupled physics imaging technique since it combines two kind of waves \cite{ArrSch12}.
As stand alone imaging techniques PAT and OCT are not capable of recovering all diagnostically relevant physical parameters, but only some combinations of them, see \cite{BalRen11a} for PAT and \cite{ElbMinSch15} for OCT.

Recently setups which combine different imaging modalities, have been investigated mathematically with the objective to reconstruct more diagnostically relevant physical parameters from the measurements. 
Particular applications are coupled physics imaging systems 
and elastography \cite{Bal12, Kuc12, WidSch12}, to name but a few. 
We refer to these techniques as hybrid imaging or multi-modal imaging systems. Note that in the mathematical literature 
the name hybrid imaging is also used for coupled physics imaging.

In this work we consider the multi-modal PAT/OCT system, developed for imaging biological tissues, see  \cite{DreLiuKumKam14, LiuCheZab16, LiuSchmSanZab14, LiuSchmSanZab13, ZhaPovLauAleHof11}. 
We show that with such a system, in contrast to the single modality setups, we obtain sufficient measurements which allow us to extract quantitative information on the electric susceptibility and the Grüneisen parameter of the sample. 
In the multi-modal PAT/OCT system, two different excitation laser systems, both operating in the same wavelength range, are used. 
The PAT and OCT scans are performed sequentially and vary a lot in acquisition times (around 5 minutes in PAT and less than 30 seconds in 
OCT). The obtained PAT and OCT images are co-registered afterwards. 

In \autoref{section_com}, we describe mathematically the multi-modal PAT/OCT setup. We use the model, from \cite{ElbMinSch17},  based on Maxwell's equations for the electric permittivity. 
In \autoref{sec_inv}, we present the equivalence of the inverse problem of recovering both optical parameters with the 
solution of a Fredholm integral equation of the first kind for the Gr\"uneisen parameter. Here the kernel of the 
integral operator depends on the PAT measurements.

We propose a numerical reconstruction method based on a Galerkin method using a series expansion of the unknown 
functions with respect to Hermite functions, see \autoref{section_numerics}. The discretization of the continuous 
integral operator results in a system of linear algebraic equations.
We solve the matrix equation using Tikhonov regularization.  
Numerical results which justify the feasibility of the proposed method are presented in \autoref{section_results}.

\section{The multi-modal PAT/OCT system}\label{section_com}

We consider the two modalities independently. Full field illumination is used in PAT and focused in OCT. 
The medium in OCT is illuminated by a Gaussian light. However, we can assume that the plane wave illumination 
is still valid \cite{Fer96}.
 
\subsection{Light propagation}

We consider macroscopic Maxwell's equations in order to model the interaction of the incoming light with the sample. These equations describe the time evolution of the electric and magnetic fields $E$ and $B$ for given charge density $\rho$ and electric current $J$:
\begin{subequations}\label{eqMaxwellMacro}
	\begin{alignat}{2}
		\div_{x} D(t,x) &= 4\pi \rho (t,x),\quad &&t\in\R,\;x\in\R^3,\label{eqMaxwellMacro1} \\
		\div_{x} B(t,x) &= 0,\quad &&t\in\R,\;x\in\R^3,\label{eqMaxwellMacro2a} \\
		\curl_{x} E(t,x) &= -\frac1c\partial_t B(t,x),\quad &&t\in\R,\;x\in\R^3,\label{eqMaxwellMacro3} \\
		\curl_{x} H (t,x) &= \frac1c\partial_t D(t,x)+\frac{4\pi}c J (t,x),\quad &&t\in\R,\;x\in\R^3 ,\label{eqMaxwellMacro4}
	\end{alignat}
\end{subequations}
where $D \equiv E + 4\pi P$ is the electric displacement and $H \equiv B - 4\pi M$ denotes the effective magnetic field, related to the electric and magnetic polarization fields $P$ and $M,$ respectively. We specify the material properties of the medium.

\begin{definition}\label{def_medium}
The medium is called non-magnetic if $M=0$, and perfect linear dielectric and isotropic if there exist a scalar function 
$\chi\in C^\infty_{\mathrm c}(\R\times\R^3;\R)$ the electric susceptibility, with $\chi(t,x)=0$ for all $t<0$, $x\in\R^3,$ (this property is referenced 
as causality), such that 
\begin{subequations}\label{eq_dielectric}
	\begin{align}
P (t,x) &=  \int_\R \chi (\tau , x) E(t-\tau, x) \d \tau  , \label{eq_polar}\\
 J(t,x) &= 0 .
	\end{align}
\end{subequations}
The electric susceptibility describes the optical properties of the medium and is the parameter to be determined. In addition, we assume that the medium has no free charges, meaning $\rho = 0$ in \eqref{eqMaxwellMacro1}.
\end{definition}

Under the assumptions \eqref{eq_dielectric} of \autoref{def_medium}, combining equations \eqref{eqMaxwellMacro3} and \eqref{eqMaxwellMacro4} we obtain the vector Helmholtz equation for the electric field
\begin{equation}\label{vector1}
\curl_{x}\curl_{x} E(t,x) +\frac1{c^2}\partial_{tt} E(t,x) = -\frac{4\pi}{c^2}  \int_\R \partial_{tt}\chi (\tau , x) E(t-\tau, x) \d \tau .
\end{equation}

Let $\Omega \subset \R^3$ denote the domain where the object is located, meaning $\supp \chi (t,\cdot) \subset \Omega$ for all $t\in\R.$ 
\begin{definition}\label{definition_initial}
We call $E^{(0)}$ an initial field if it satisfies the wave equation
\begin{equation}\label{wave_incident}
\Delta_x E^{(0)} (t,x)-\frac1{c^2}\partial_{tt} E^{(0)}(t,x) = 0,
\end{equation}
and $\div_{x} E^{(0)} = 0,$ and
does not interact with the medium until the time $t=0,$ meaning
\[
\supp E^{(0)} (t,\cdot) \cap \Omega = \emptyset \quad \text{for all} \quad t<0 .
\]
\end{definition}

The initial pulse $E^{(0)}$ is a vacuum solution of Maxwell's equations, meaning it satisfies \eqref{vector1} with $\chi \equiv 0.$ Indeed, using the vector identity $\curl_x\curl_x E= \mbox{grad}_{x}\div_x E-\Delta_x E$ in (\ref{vector1}) we obtain the wave equation (\ref{wave_incident}), since $\div_x E^{(0)} = 0.$

Then, if the medium is given by \autoref{def_medium}, we consider $E$ as the solution of \eqref{vector1} with initial condition
\begin{equation}\label{eq_initial}
E(t,x) = E^{(0)} (t,x) \quad \text{for all} \quad t<0 , \, x\in \R^3 .
\end{equation}

\subsubsection{Specific illumination}
For both imaging modalities we consider the same incoming field. We use the convention
\[
\hat f (\omega ,x) = \int_\R f(t,x) \e^{\i\omega t} \d t, 
\] 
for the Fourier transform of a integrable function $f$ with respect to time $t.$

The  multiple laser pulses centered around different frequencies $\nu$, are described by the initial electric fields 
\begin{equation}\label{eqLocalisedPulse1}
E^{(0)}_\nu(t,x)=f_\nu(t+\tfrac{x_3}c)\eta, \quad \nu>0,
\end{equation}
which describe linearly polarized plane waves moving in the direction $-e_3 ,$  for some $f_\nu \in C^\infty_{\mathrm c}(\R)$ and fixed polarization vector $\eta\in\R^2\times\{0\},$ with $\vert\eta \vert =1.$ These fields satisfy \eqref{wave_incident} for every $\nu$.
We assume that the Fourier transform of $f_\nu$ satisfies
\begin{equation}\label{eqLocalisedPulse}
\supp\hat f_\nu\subset[-\nu-\varepsilon,-\nu+\varepsilon]\cup[\nu-\varepsilon,\nu+\varepsilon] ,
\end{equation}
for some sufficiently small $\varepsilon>0$. We denote by $E_\nu$ the solution of (\ref{vector1}) for the specific initial field $E^{(0)}_\nu .$
The multiple illuminations result to multi-frequency PAT measurements, but they do not provide extra information in OCT, see \cite[Lemma 3.6]{ElbMinSch17}.

\subsection{PAT measurements}

Let the medium be defined as in Definition \ref{def_medium}. 
Then, we estimate the averaged change in energy density around a point $x,$ for every $\nu,$ by 
\begin{equation}\label{eqAbsorbedEnergy}
\partial_t\mathcal E_\nu (t,x) \approx \left<E_\nu(t,x),\partial_t  P_\nu (t,x)\right>.
\end{equation}
In order to derive the above formula we have to consider the interaction of the medium with the incoming electromagnetic wave locally. For a derivation, using microscopic Maxwell's equations, see for instance \cite[Section 4]{ElbMinSch17}.

The laser pulse is absorbed by the medium and part of it is transformed into heat. This generates a pressure wave which is then measured on the object surface. Since the laser pulse is typically very short, the propagation of the acoustic wave during thermal absorption can be neglected. Then,  we consider as PAT measurements the initial pressure density $p$ which is proportional to the absorbed energy
\begin{equation}\label{pat_data}
p_\nu(x) = \Gamma (x) \int_\R \partial_t\mathcal E_\nu (\tau ,x) \d \tau .
\end{equation}

The proportionality factor $\Gamma$ is the Gr\"uneisen parameter, a parameter which, together with the susceptibility $\chi,$ describes the optical properties of our medium.

\subsection{OCT measurements}

 In the frequency domain, the equation \eqref{vector1} and the condition \eqref{eq_initial} result to an integral equation of Lippmann-Schwinger type \cite{ColKre98, ElbMinSch15}. 

\begin{lemma}
Let the medium be defined as in \autoref{def_medium} and $E_\nu^{(0)}$  as in \autoref{definition_initial}. If $E_\nu$ is a solution of \eqref{vector1} with initial values \eqref{eq_initial}, then its Fourier transform solves the Lippmann-Schwinger integral equation
\begin{equation}\label{Lippmann}
\hat E_\nu (\omega ,x) = \hat E_\nu^{(0)}(\omega,x) + \left(\frac{\omega^2}{c^2}+\grad_x\div_x\right)\int_{\R^3}\frac{\e^{\i\frac\omega c|x-y|}}{ |x-y|}\hat\chi (\omega,y)\hat E_\nu (\omega,y)\d y .
\end{equation}
\end{lemma}

Due to the limiting penetration depth of OCT (1 to 2 millimeters), the medium can be considered as weakly scattering, since only single scattering events will be measured. In addition, in OCT the measurements are performed in a distance much larger compared to the size of the medium. 

The Born approximation allows us to obtain an explicit form for $\hat E_\nu$ from the Lippmann-Schwinger equation \eqref{Lippmann}. In the limiting case $\hat\chi \rightarrow 0,$ we take the first order approximation of the electric field by replacing $\hat E_\nu$ with $\hat E_\nu^{(0)}$ in the integrand of \eqref{Lippmann}.

We write $x$ in spherical coordinates $x = \rho \vartheta , \, \rho >0, \, \vartheta \in S^2 .$ Under the far-field approximation, 
we consider the  asymptotic behavior of the expression \eqref{Lippmann} for $\rho\rightarrow \infty,$ uniformly in $\vartheta$ 
\begin{equation*}
	\hat E_\nu (\omega,\rho\vartheta) \simeq\hat E_\nu^{(0)}(\omega,\rho\vartheta)- \e^{\i\frac\omega c\rho} \frac{\omega^2}{\rho c^2}\int_{\R^3}\vartheta\times\big(\vartheta\times(\hat\chi(\omega,y)\hat E_\nu (\omega,y))\big)\e^{-\i\frac\omega c\left<\vartheta,y\right>}\d y.
\end{equation*}
Then, we define
\begin{equation}\label{eqFarField}
	\hat E_\nu^{(1)}(\omega,\rho\vartheta) := \hat E_\nu^{(0)}(\omega,\rho\vartheta)- \e^{\i\frac\omega c\rho} \frac{\omega^2}{\rho c^2}\int_{\R^3}\vartheta\times\big(\vartheta\times(\hat\chi(\omega,y)\hat E_\nu^{(0)}(\omega,y))\big)\e^{-\i\frac\omega c\left<\vartheta,y\right>}\d y,
\end{equation}
as the electric field considering both approximations.

The approximated backscattered light $\hat E_\nu^{(1)}- \hat E_\nu^{(0)}$ is combined with a known back-reflected field 
and its correlation is measured at each point on the detector surface. Under some assumptions on the incident 
illumination we state that what we actually measure in OCT is the backscattered light at a detector placed far from the 
medium \cite[Proposition 8]{ElbMinSch15}.

Then, we formulate the direct problem as:

\begin{definition}[direct problem]
Given a medium as in \autoref{def_medium} with susceptibility $\chi$ and Gr\"uneisen parameter $\Gamma,$ and incident illumination $E_\nu^{(0)}$ of the form \eqref{eqLocalisedPulse1}, the direct problem is to find  the PAT measurements $p_\nu (x), \, x \in \Omega, \, \nu >0,$ given by \eqref{pat_data}, and the OCT measurements $$(\hat E_\nu^{(1)}-\hat E_\nu^{(0)}) (\omega,\rho\vartheta), \, \omega\in\R\setminus\{0\}, \, \vartheta\in S^2_+ =\{\vartheta\in S^2\mid\vartheta_3>0\},\, \nu >0,$$ given by \eqref{eqFarField}.
\end{definition}

\section{The Inverse Problem}\label{sec_inv}

In the following the  assumptions on the medium (\autoref{def_medium}) hold and especially the causality of $\chi.$
We denote by $\tilde\chi $ the three-dimensional Fourier transform of $\hat\chi$ with respect to space
	\[
	\tilde\chi(\omega,k ) = \int_{\R^3} \hat \chi(\omega,x ) \e^{-\i \langle k,x\rangle} \d x.
	\]

The OCT system, by replacing $E_\nu^{(0)}$ in (\ref{eqFarField}) and simple calculations, see \cite[Proposition 9]{ElbMinSch15}, provide us with the data
\begin{equation}\label{eqMeasurementDataIsotropic}
\tilde\chi (\omega,\tfrac\omega c(\vartheta+e_3)),\quad\omega\in\R\setminus\{0\},\;\vartheta\in S^2_+ .
\end{equation}
However, in practice, these data are incomplete because of the band-limited source and size of the detector. Thus, we get the spatial and temporal Fourier transform of $\chi$ only in a subset of $\R \times \R^3 .$

Then, the inverse problem we address here reads:
\begin{definition}[inverse problem]\label{def_inverse}
Given a medium as in \autoref{def_medium} and incident fields $E^{(0)}_\nu$ of the form \eqref{eqLocalisedPulse1} for all $\nu >0,$ the inverse problem is to recover the parameters $\hat\chi$ and $\Gamma$ given the internal PAT measurements $p_\nu(x),$ for $x\in \Omega,$ and all $\nu>0,$ given by  \eqref{pat_data}, and the external OCT data  $\tilde\chi (\omega,\tfrac\omega c(\vartheta+e_3)), \, \omega\in\R\setminus\{0\}, \, \vartheta\in S^2_+ ,$ given by \eqref{eqMeasurementDataIsotropic}.
 \end{definition}

Similar inverse problems have been considered in \cite{BalRenUhlZho11, BalZho14} where the far-field measurements from OCT are replaced by boundary measurements and in \cite{BalUhl12} for the diffusion approximation of the radiative transfer equation.

To present an equivalent formulation of the inverse problem, we assume that in both imaging techniques, we illuminate with multiple laser pulses with small spectrum centered around different frequencies. This setup describes swept-source OCT and multi-frequency PAT measurements.

 First we describe the PAT measurements for multiple laser pulses. We
combine (\ref{eqAbsorbedEnergy})
and  (\ref{pat_data}) to get
\[ 
p_\nu (x) = \Gamma(x)\int_\R \left<E_\nu(t,x),\partial_t P_\nu(t,x)\right>\d t , 
\]
where $E_\nu$ is the electric field generated by the laser pulse $E^{(0)}_\nu .$ Using the Fourier transform of \eqref{eq_polar} we derive 
\begin{equation*}
p_\nu(x)=\Gamma(x)\frac1{2\pi} \int_\R -\i\omega \hat\chi(\omega, x)|\hat E_\nu(\omega,x)|^2\d\omega.
\end{equation*}

\begin{remark}
In the case of nonlinear medium, the polarization field $P$ is usually expressed as a power series of the electric field $E.$ Then, the third order term contributes to the so-called two-photon absorbed energy \cite{Fri82}. We refer to \cite{RenZha16} for reconstructions in two-photon PAT.
\end{remark}

As in OCT, we replace $E_\nu$ by the initial pulse $E_\nu^{(0)}$ and  we approximate the PAT data by 
\[
 p_\nu(x) \approx \Gamma(x)\frac1{2\pi} \int_\R -\i\omega \hat\chi(\omega, x)|\hat f_\nu (\omega)|^2\d\omega, 
 \]
since $|
\eta |=1.$  The support of $\hat f_\nu$ is localized around the frequency $\nu$, see \eqref{eqLocalisedPulse}. Thus, we get in the limit $\varepsilon\to0$ (for constant norm $\|\hat f_\nu\|_2$) that
\begin{equation*}
p_\nu(x) \simeq \frac1{2\pi}\|\hat  f_\nu\|_2^2\Gamma(x)(-\i\nu \hat\chi(\nu,x)+\i \nu \hat\chi (-\nu,x)) = \frac1{\pi}\|\hat f_\nu\|_2^2 \Gamma(x) \nu \Im(\hat\chi(\nu,x)).
\end{equation*}

We define $p ( \nu , x) := \frac{\pi}{\nu} \|\hat f_\nu \|_2^{-2} \, p_\nu (x).$ Then, we get asymptotically
\begin{equation}\label{eqPATOCTmeasSimplified}
\boxed{p(\nu ,x) \simeq \Gamma(x)  \Im(\hat\chi(\nu ,x)). }
\end{equation}

We assume measurements for all frequencies $\nu>0.$ Recall that $\chi$ is a causal real valued function. Then the real part of $\hat\chi$ can be completely determined from the imaginary part via the Kramers--Kronig relation
\begin{equation*}
\Re(\hat\chi (\omega,x)) = \mathcal{H}[\Im \hat\chi ] (\omega,x) .
\end{equation*}
Here, $\mathcal{H}$ denotes the Hilbert transform with respect to frequency
\[
\mathcal{H}[ f] (\omega ,x ) = \frac{1}{\pi} \int_{\R} \frac{ f(\tilde{\omega},x)}{\tilde{\omega}-\omega} \d \tilde{\omega}.
\]

We use the above two equations in order to describe the OCT data \eqref{eqMeasurementDataIsotropic}. Then we end up with the Fredholm integral equation
\begin{equation}\label{com_fre}
\boxed{\int_{\R^3} \left(\mathcal{H} [p] (\omega , y) + \i p (\omega , y) \right) \, \e^{-\i \tfrac{\omega}c \langle \vartheta + e_3,  y\rangle} \frac1{\Gamma (y)} \d y =   \tilde\chi (\omega,\tfrac\omega c(\vartheta+e_3)) ,}
\end{equation}

for the Gr\"uneisen parameter $\Gamma.$ Once \eqref{com_fre} is solved, we can easily recover the imaginary part of $\hat\chi$ from equation~\eqref{eqPATOCTmeasSimplified}. 

\begin{remark} If the medium is a perturbation of a single material then the above equation is transformed to a Fredholm integral equation of the second kind for a new function depending on $\tfrac1\Gamma$ \cite{ElbMinSch17}. At least in this simplified setting, we find that using the multi-modal model PAT/OCT we can (uniquely) determine the Gr\"uneisen parameter and the susceptibility $\chi$ describing the absorption and scattering properties of the medium.
\end{remark}

Observing the formulas \eqref{eqPATOCTmeasSimplified} and \eqref{com_fre}, we rewrite the inverse problem (\autoref{def_inverse}) in its simplified form:

\begin{definition}[simplified inverse problem]
 Find $\Gamma (x)$ and $\hat \chi (\omega,x),$ given $\tilde\chi (\omega,\tfrac\omega c(\vartheta+e_3)),$ for all $\omega\in\R\setminus\{0\},\;\vartheta\in S^2_+ $ (approximated OCT data) and the product $\Gamma (x) \Im(\hat\chi(\omega,x)),$ for all $\omega\in\R\setminus\{0\},\;x\in \Omega$ (approximated PAT data).
\end{definition}

In the following section we present a Galerkin type method for the numerical solution of equation \eqref{com_fre} considering two types of media.  There exist also other projection methods for the numerical solution of integral equations, the collocation method and the method of moments and quadrature methods, like the  Nystr\"om method \cite{Hac95, Kre99, PolMan98}.

\section{Numerical Implementation}\label{section_numerics}
Without loss of generality we set $c=1$ and we specify $\Omega =[-l,l]^3 .$ 
For the numerical examples we have to introduce the parameter $\tilde{\Gamma},$ related to the physical parameter $\Gamma,$ which satisfies
\[
\tilde{\Gamma} (x) = \frac{1}{\Gamma(x)}, \quad \mbox{for } x\in \Omega, \quad \mbox{and} \quad \supp \tilde{\Gamma} \subset \Omega_L ,
\]
where $\Omega_L =[-L,L]^3,$ for $L>l.$ This is possible since $\Gamma \geq \Gamma_0 >0,$ with $\Gamma_0 \sim 1$ in biological tissues. 
In addition, we do not consider the restrictions on the frequency in the following analysis. 

\subsection{Medium with depth-dependent coefficients}\label{section_depth}
In the first example, we assume that both parameters $\Gamma$ and $\hat\chi$ are only depth-dependent, meaning that they vary only in the incident direction. 
  Then, $\tilde\Gamma$ and $\hat\chi$ admit the forms
\begin{equation*}
\tilde\Gamma (x) = \id_{[-L,L]^2} (x_1 ,x_2) \gamma (x_3)  , \quad \text{and} \quad  \hat\chi (\omega , x) = \id_{[-l,l]^2} (x_1 ,x_2) \psi (\omega ,x_3 ) ,
\end{equation*}
respectively. Here $\id$ denotes the characteristic function. This case represents media which have a multilayer structure with depth-dependent properties, like the human skin.
If the illumination is focused to a small region inside the object and this region is small enough such that the functions can be assumed constant in both directions $e_1$ and $e_2,$ we get the above forms.

Thus the problem reduces to the problem of recovering a one-dimensional function $\gamma$. We do not consider the two-dimensional detector array but only the measurements at the single point detector located at the position $(0,\,0,\,d),$ meaning we set $\vartheta = e_3. $ Then, the equation \eqref{com_fre} takes the simplified form
\begin{equation}\label{foc_fredholm}
\int_{\R} \left(\mathcal{H} [p] (\omega , y_3) + \i p (\omega , y_3) \right)  \e^{- \i 2\omega  y_3 } \gamma (y_3) \d y_3 =  m (\omega)  ,
\end{equation}
where $m (\omega) := (\tfrac{\pi}{l})^2  \tilde\chi (\omega,2 \omega e_3 ). $ 

Let $p \in (L^2 (\R))^2$ and $\gamma \in L^2 (\R). $  Since the kernel of the integral operator and the right-hand side in \eqref{foc_fredholm} have specific structures containing Hilbert and Fourier transforms we consider as orthonormal basis of $L^2 (\R)$ the Hermite functions $h_k , \, k\in \N_0$. In addition the multi-dimensional Hermite functions can be written as sum of products of the usual Hermite functions. Their properties are given in \autoref{section_her}.
Other choices are also possible, especially when we treat the three-dimensional problem with real data, for instant using wavelets as basis functions.

Let $x\in \R.$ The Hermite polynomials are defined by the formula
\begin{equation*}
H_k (x) = (-1)^k \frac{d^k}{dx^k} (\e^{-x^2}) \e^{x^2}, \quad k \in \N_0 .
\end{equation*} 
The normalized Hermite functions are given by
\begin{equation}\label{her_fun}
h_k (x) = \alpha_k H_k (x) \e^{-\tfrac12 x^2},  k \in \N_0  ,
\end{equation}
where $\alpha_k = (2^k k! \sqrt{\pi})^{-\tfrac12}.$ The functions $h_k$ satisfy the orthonormality condition
\begin{equation*}
\int_\R h_k (x) h_l (x) \d x = \delta_{k,l} .
\end{equation*}

\begin{proposition}\label{propo_1d}
Let $x\in\R.$ We consider the expansion
\begin{equation}\label{foc_expan_gamma}
\gamma (\tfrac{x}2) = \sum_{k=0}^\infty \gamma_k h_k (x) ,
\end{equation}
with coefficients $\gamma_k \in \R, \, k \in \N_0$  and 
\begin{equation}\label{foc_expan_pat}
p (\omega, \tfrac{x}2) = \sum_{k,l=0}^\infty p_{k,l} h_k (\omega) h_l (x) ,
\end{equation}
with coefficients 
$ p_{k,l} \in \R,\, k,l \in \N_0 . $
Then, if $\gamma$ satisfies the integral equation \eqref{foc_fredholm}, the coefficients $\gamma_k ,\, k\in\N_0$ solve the equation
\begin{equation}\label{1d_full_7}
\sum_{j=0}^\infty \gamma_j A_{j,s} = m_s , \quad s\in \N_0 ,
\end{equation}
where
\begin{equation*}
\begin{aligned}
m_s &=  \int_\R 2 \e^{
\tfrac{\omega^2}4} m (\omega) h_s (\omega) \d \omega,  \\
A_{j,s} &:=  \sum_{k,l=0}^\infty (\tilde p_{k,l} +\i p_{k,l} )  \sum_{n=0}^{\min (j,l)} \beta_{j,l,n} \,\zeta_{j+l-2n}\frac{(j+l-2n)!}{2^{j+l-2n}} 
\sum_{r=0}^{\left[ \tfrac{j+l-2n}{2}\right] } \frac{1}{r!q! \,\alpha_q} 
\\
&\phantom{=} \times \sum_{s=0}^\infty \id_{[\vert k-q\vert,k+q ]}(s) \beta_{k,q,\tfrac{k+q-s}2}  , \quad \mbox{and } q := j+l-2n-2r. 
\end{aligned}
\end{equation*}

\end{proposition}

Before proving this Proposition, we state the following lemma. Its proof is presented in \autoref{section_her}.
\begin{lemma}\label{lemma1}
Let $k\in \N_0 .$ Then
\begin{equation}\label{eq_lemma1}
\int_\R \e^{-\tfrac{x^2}2} h_k (x) \e^{- \i \omega  x } \d x = \zeta_k \e^{-
\tfrac{\omega^2}4} \omega^k ,
\end{equation}
where $\zeta_k = 2\sqrt{\pi^3 } (-\i)^k \alpha_k ,$ for $\alpha_k$ as in \eqref{her_fun}.
\end{lemma}

\begin{proof}[\autoref{propo_1d}]
Using the expansion \eqref{foc_expan_pat} and considering \eqref{her_hil}, we get
\begin{equation*}
\mathcal{H} [p] (\omega , x) + \i p (\omega , x) = \sum_{k,l=0}^\infty (\tilde p_{k,l} +\i p_{k,l} ) h_k (\omega) h_l (2x).
\end{equation*}
The coefficients $\tilde p_{k,l},$ using \eqref{her_coe_hil}, are given by
\begin{equation*}
\tilde p_{k,l}  = (-\i)^{k+1} \sum_{m=0}^\infty p_{m,l} (-\i)^m \int_\R \sign (\omega) h_k (\omega) h_m (\omega) \d \omega .
\end{equation*}

We substitute the above expansions and \eqref{foc_expan_gamma} in \eqref{foc_fredholm} and we obtain 
\begin{equation*}
\sum_{j=0}^\infty \gamma_j \sum_{k,l=0}^\infty (\tilde p_{k,l} +\i p_{k,l} ) h_k (\omega) \int_\R h_j (2y_3) h_l (2y_3) \e^{- \i 2\omega  y_3 } \d y_3 =  m (\omega) .
\end{equation*}
We rewrite the product of the two Hermite functions in the integrand using the formula \eqref{her_pro_fun} and we change variables to get
\begin{equation}\label{1d_full_2}
\sum_{j=0}^\infty \gamma_j \sum_{k,l=0}^\infty (\tilde p_{k,l} +\i p_{k,l} ) h_k (\omega) \sum_{n=0}^{\min (j,l)} \beta_{j,l,n}\int_\R \e^{-\tfrac{x^2}2} h_{j+l-2n} (x) \e^{- \i \omega  x } \d x = 2 m (\omega) .
\end{equation}

Then, equation \eqref{1d_full_2} using \eqref{eq_lemma1} takes the form
\begin{equation}\label{1d_full_3}
\sum_{j=0}^\infty \gamma_j \sum_{k,l=0}^\infty (\tilde p_{k,l} +\i p_{k,l} ) h_k (\omega) \sum_{n=0}^{\min (j,l)} \beta_{j,l,n} \,\zeta_{j+l-2n} \omega^{j+l-2n} = \tilde m (\omega) ,
\end{equation}
where $\tilde m (\omega )= 2 \e^{
\tfrac{\omega^2}4} m (\omega).$ 

Using \eqref{her_fun} and  \eqref{her_inv}, we get
 \begin{equation}\label{her_inv_fun}
\omega^k = \frac{k!}{2^k} \e^{\tfrac{\omega^2}2}\sum_{q=0}^{\left[ \tfrac{k}{2}\right] } \frac{1}{q!(k-2q)! \alpha_{k-2q}}  h_{k-2q} (\omega) .
\end{equation}

We substitute this expansion in \eqref{1d_full_3} to obtain
\begin{multline}\label{1d_full_5}
\sum_{j=0}^\infty \gamma_j \sum_{k,l=0}^\infty (\tilde p_{k,l} +\i p_{k,l} )  \sum_{n=0}^{\min (j,l)} \beta_{j,l,n} \,\zeta_{j+l-2n}\frac{(j+l-2n)!}{2^{j+l-2n}} \e^{\tfrac{\omega^2}2} \\
\times \sum_{r=0}^{\left[ \tfrac{j+l-2n}{2}\right] } \frac{1}{r!q! \alpha_q} h_k (\omega) h_q (\omega)= \tilde m (\omega) ,
\end{multline}
where for simplicity we set $q: = j+l-2n-2r .$ Again the last product using \eqref{her_pro_fun} admits the form
\[
h_k (\omega) h_q (\omega)= \e^{-\tfrac{\omega^2}2} \sum_{u=0}^{\min (k,q)} \beta_{k,q,u} h_{k+q-2u} (\omega) .
\]

We expand also the data using the same basis functions
\begin{equation*}
\tilde m (\omega) = \sum_{k=0}^\infty m_k h_k (\omega), \quad \text{for} \quad m_k = \int_\R \tilde m (\omega) h_k(\omega) \d 
\omega  ,
\end{equation*}
and in order to obtain a linear equation for $\gamma_j$ we have to enlarge the index of the last sum. We set $s: = k+q-2u$ and for $\tfrac{k+q-s}2 \in \N_0$ we reformulate \eqref{1d_full_5} using the above formulas as
\begin{multline*}
\sum_{j=0}^\infty \gamma_j \sum_{k,l=0}^\infty (\tilde p_{k,l} +\i p_{k,l} )  \sum_{n=0}^{\min (j,l)} \beta_{j,l,n} \,\zeta_{j+l-2n}\frac{(j+l-2n)!}{2^{j+l-2n}} 
\sum_{r=0}^{\left[ \tfrac{j+l-2n}{2}\right] } \frac{1}{r!q! \,\alpha_q} \\ \times \sum_{s=0}^\infty \id_{[\vert k-q\vert,k+q ]} (s) \beta_{k,q,\tfrac{k+q-s}2} h_s (\omega) = \sum_{s=0}^\infty m_s h_s (\omega) .
\end{multline*}

Equating the coefficients in the above equation yields \eqref{1d_full_7}.
\end{proof}

The final step, for the Galerkin method, is to consider a finite dimensional subset of $L^2 (\R ),$ meaning restrict ourselves to a finite number of coefficients. Let $j,k,l = 0,...,N-1 .$ Then the definitions used in the above analysis gives $q = 0,...,2(N-1)$ and $s=0,...,3(N-1).$ Finally, the discrete linear system of \eqref{1d_full_7} reads
\begin{equation}\label{eq_final_1d}
\b A \bm \gamma = \b m ,
\end{equation}
where $\b A = (A_{s,j}) \in \C^{(3N-2)\times N}, \bm \gamma =(\gamma_j)\in \R^N$ and $\b m = (m_s)\in \C^{3N-2}.$

\subsection{Medium with coefficients constant in one direction}\label{section_2d}
In this case, we assume that both parameters are constant only in one direction, let us say in $e_1 .$  Then, $\tilde\Gamma$ and $\hat\chi$ take the forms
\begin{equation*}
\tilde\Gamma (x)  = \id_{[-L,L]} (x_1 ) \gamma (x_2, x_3)  , \quad \text{and} \quad  \hat\chi (\omega , x) = \id_{[-l,l]} (x_1 ) \psi (\omega , x_2, x_3 ) ,
\end{equation*}
respectively. This assumption results to a two-dimensional function $\gamma,$ a case more involved compared to \autoref{section_depth} that approximates better the unconditional general problem.
Here, we need two-dimensional data, thus we have to consider measurements for all frequencies in a one-dimensional array, modeling measurement points on a line. 

The equation \eqref{com_fre}, for $c=1,$ now takes the form
\begin{equation}\label{foc_fredholm_2d}
\int_{\R} \int_{\R} \left(\mathcal{H} [p] (\omega , y_2 , y_3) + \i p (\omega ,y_2 , y_3) \right)  \e^{- \i \omega (\vartheta_2 y_2 + \tilde\vartheta_3 y_3 )} \gamma (y_2 , y_3) \d y_2  \d y_3 =  m ( \omega, \vartheta )  ,
\end{equation}
where $\tilde\vartheta_3 = \vartheta_3 +1,$  and $m (\omega, \vartheta ) := \frac{\pi}{l}  \tilde\chi (\omega,\tfrac\omega c(\vartheta+e_3)). $

\begin{proposition}\label{propo_2d}
Let $x=(x_1 , x_2)\in\R^2.$ We use (\ref{her_3d}), for $\bm k = (k,\,l)$, and we expand $\gamma$ as
\begin{equation}\label{foc_expan_gamma_2d}
\gamma (x) = \sum_{\bm k=0}^\infty \gamma_{\bm k} h_{\bm k} (x) = \sum_{k,l=0}^\infty \gamma_{ k,l} h_{k} (x_1) h_l (x_2),
\end{equation}
where the coefficients $\gamma_{k,l}$ are defined by
\begin{equation*}
\gamma_{k,l} = \int_\R \int_\R \gamma (x_1 ,x_2) h_k (x_1 ) h_l (x_2 ) \d x_1 \d x_2 ,\quad k,l \in \N_0 .
\end{equation*}
and we assume the expansion
\begin{equation}\label{foc_expan_pat_2d}
p (\omega, x) = \sum_{k,l,a=0}^\infty p_{k,l,a} h_k (\omega)   h_l (x_1) h_a (x_2) ,
\end{equation}
where 
\[
p_{k,l,a} = \int_{\R}\int_{\R}\int_{\R} p (\omega, x_1 ,x_2) h_k (\omega) h_l (x_1) h_a (x_2) \d \omega  \d x_1 \d x_2 , \quad k,l,a \in \N_0 
\]
Then, if $\gamma$ solves the integral equation (\ref{foc_fredholm_2d}), its coefficients $\gamma_{k,l}, \, k,l\in N_0$ satisfy the equation
\begin{equation}\label{2d_full_7}
\sum_{k,l=0}^\infty \gamma_{k,l} B_{k,l,\mu} (\vartheta) = m_\mu (\vartheta), \quad \mu \in \N_0,
\end{equation}
for 
\begin{equation*}
\begin{aligned}
 m_\mu (\vartheta) &= \int_\R \e^{\frac{\omega^2 \tilde\vartheta_3 }2 } m (\vartheta , \omega) h_\mu (\omega) \d \omega ,\\
 B_{k,l,\mu} (\vartheta) &=  \sum_{a,n,u=0}^\infty (\tilde p_{a,n,u} +\i p_{a,n,u} )   \sum_{r=0}^{\min (k,n)} \beta_{k,n,r}
\, \zeta_{k+n-2r}   \vartheta_2^{k+n-2r} \\
&\phantom{=}\times \sum_{q=0}^{\min (l,u)} \beta_{l,u,q} \,\zeta_{l+u-2q}   \tilde\vartheta_3^{l+u-2q} 
\frac{s!}{2^s} \sum_{j=0}^{\left[ \frac{s}{2}\right] } \frac{1}{j!(s-2j)! \alpha_{s-2j}}  \\ &\phantom{=}\times \id_{[\vert a-s+2j \vert,  a+s-2j]} (\mu) \beta_{a,s-2j,\frac{a+s-2j-\mu}2} \,,
\end{aligned}
\end{equation*}
with $s:= k+n+l+u-2r-2q.$
\end{proposition}

\begin{proof}
The equation (\ref{foc_fredholm_2d}) using the expansions (\ref{foc_expan_gamma_2d}) and (\ref{foc_expan_pat_2d}) results to 
\begin{multline*}
 \sum_{k,l=0}^\infty \gamma_{k,l} \sum_{a,n,u=0}^\infty (\tilde p_{a,n,u} +\i p_{a,n,u} ) h_a (\omega)
 \int_\R \int_\R  h_k (y_2 ) h_l (y_3) h_n (y_2 ) \\
\times h_u (y_3)  \e^{- \i \omega (\vartheta_2 y_2 + \tilde\vartheta_3 y_3 )}  \d y_2  \d y_3 =  m ( \omega, \vartheta) .
\end{multline*}

We apply twice the formula (\ref{her_pro_fun}) for the product of two Hermite functions, to obtain
\begin{multline*}
\sum_{k,l=0}^\infty \gamma_{k,l} \sum_{a,n,u=0}^\infty (\tilde p_{a,n,u} +\i p_{a,n,u} ) h_a (\omega) \sum_{r=0}^{\min (k,n)} \beta_{k,n,r}
\int_\R \e^{-\frac{y_2^2}2} h_{k+n-2r} (y_2) \e^{- \i \omega \vartheta_2 y_2} \d y_2
\\
\times \sum_{q=0}^{\min (l,u)} \beta_{l,u,q}
\int_\R \e^{-\frac{y_3^2}2} h_{l+u-2q} (y_3) \e^{- \i \omega \tilde\vartheta_3 y_3} \d y_3  =  m ( \omega, \vartheta) .
\end{multline*}

The last two integrals can be again simplified using lemma \ref{lemma1}. We get
\begin{multline*}
 \sum_{k,l=0}^\infty \gamma_{k,l} \sum_{a,n,u=0}^\infty (\tilde p_{a,n,u} +\i p_{a,n,u} ) h_a (\omega) \sum_{r=0}^{\min (k,n)} \beta_{k,n,r}
\zeta_{k+n-2r} \, \e^{-
\frac{(\omega \vartheta_2 )^2}4} (\omega \vartheta_2)^{k+n-2r}
\\
\times \sum_{q=0}^{\min (l,u)} \beta_{l,u,q} \zeta_{l+u-2q} \, \e^{-
\frac{(\omega \tilde\vartheta_3 )^2}4} (\omega \tilde\vartheta_3)^{l+u-2q}
  =  m ( \omega, \vartheta),
\end{multline*}
which for $s:= k+n+l+u-2r-2q,$ can be rewritten as
\begin{multline*}
\sum_{k,l=0}^\infty \gamma_{k,l} \sum_{a,n,u=0}^\infty (\tilde p_{a,n,u} +\i p_{a,n,u} ) h_a (\omega) \sum_{r=0}^{\min (k,n)} \beta_{k,n,r}
\zeta_{k+n-2r}   \vartheta_2^{k+n-2r}
\\
\times \sum_{q=0}^{\min (l,u)} \beta_{l,u,q} \zeta_{l+u-2q}   \tilde\vartheta_3^{l+u-2q} \omega^s
  =  \tilde m ( \omega, \vartheta),
\end{multline*}
where $\tilde{m} = \e^{\frac{\omega^2 \tilde\vartheta_3 }2 } m,$ using that $|\vartheta |= 1$ and $\vartheta_1 = 0.$ The term $\omega^s$ can be analysed using the inverse explicit expression (\ref{her_inv_fun}) resulting to
\begin{eqnarray}\label{2d_full_5}
 \sum_{k,l=0}^\infty \gamma_{k,l} \sum_{a,n,u=0}^\infty (\tilde p_{a,n,u} +\i p_{a,n,u} ) h_a (\omega) \sum_{r=0}^{\min (k,n)} \beta_{k,n,r}
\zeta_{k+n-2r}   \vartheta_2^{k+n-2r}
\sum_{q=0}^{\min (l,u)} \beta_{l,u,q} \zeta_{l+u-2q}  \nonumber\\
\times \tilde\vartheta_3^{l+u-2q} \frac{s!}{2^s} \e^{\frac{\omega^2}2}\sum_{j=0}^{\left[ \frac{s}{2}\right] } \frac{1}{j!(s-2j)! \alpha_{s-2j}}  h_{s-2j} (\omega)
  =  \tilde m ( \omega, \vartheta).
\end{eqnarray}

We expand again the product $h_a (\omega)h_{s-2j}(\omega)$ using (\ref{her_pro_fun})  as
\[
h_a (\omega)h_{s-2j}(\omega) = \e^{-\frac{\omega^2}2}  \sum_{t=0}^{\min (a,s-2j)} 
\beta_{a,s-2j,t} h_{a+s-2j-2t} (\omega).
\]
We set $\mu:= a+s-2j-2t$ and for $\frac{a+s-2j-\mu}2 \in \N_0$ the above sum can be rewritten as
\[
 \sum_{t=0}^{\min (a,s-2j)} 
\beta_{a,s-2j,t} h_{a+s-2j-2t} (\omega) = \sum_{\mu=0}^\infty \id_{[\vert a-s+2j \vert,  a+s-2j]} (\mu)\beta_{a,s-2j,\frac{a+s-2j-\mu}2} h_\mu (\omega) .
\]
We expand the right-hand side of (\ref{2d_full_5}) using the same basis functions
\begin{equation*}
\tilde m ( \omega, \vartheta) = \sum_{\mu=0}^\infty m_\mu (\vartheta) h_\mu (\omega), \quad \mbox{for} \quad m_\mu (\vartheta) = \int_\R \tilde m ( \omega, \vartheta) h_\mu(\omega) \d 
\omega  .
\end{equation*}
Then, the equation (\ref{2d_full_5}) using the above formulas and equating the coefficients results in equation (\ref{2d_full_7}).
\end{proof}

Let $k,l,a = 0,...,N-1 ,$ then we get $s = 0,...,4(N-1)$ and $\mu=0,...,5(N-1).$ Thus, the discrete linear system of \eqref{2d_full_7} admits the form
\begin{equation}\label{2d_full_final}
\bm \Gamma \, \b B (\vartheta )  = \b m (\vartheta), \quad \vartheta \in S^2_+ ,
\end{equation}
for the matrix-valued unknown function $\bm \Gamma =(\gamma_{k,l})\in \R^{N\times N},$
where $\b B  = (B_{k,l,\mu}) \in \C^{ N\times N\times (5N-4)}$  and $\b m = (m_\mu)\in \C^{5N-4}.$  
To bring the above equation into a form similar to \eqref{eq_final_1d}, we define the vector
\[
\bm\zeta = (\gamma_{0,0},...,\gamma_{0,N-1},\gamma_{1,0},...,\gamma_{1,N-1},...,\gamma_{N-1,0},...,\gamma_{N-1,N-1})^\top \in \R^{N^2},
\]
and we rearrange $\b B$ to create the matrix $\b C \in \C^{(5N-4)\times N^2}$ given by
\[  
\b C = \begin{pmatrix}
B_{0,0,0} & \hdots & B_{0,N-1,0} &   \hdots &  \hdots & B_{N-1,0,0} & \hdots & B_{N-1,N-1,0} \\
B_{0,0,1} & \hdots & B_{0,N-1,1} &  \hdots  & \hdots & B_{N-1,0,1} & \hdots & B_{N-1,N-1,1} \\
\vdots & & \vdots & &  & \vdots & & \vdots \\
B_{0,0,5(N-1)} & \hdots & B_{0,N-1,5(N-1)} &   \hdots  & \hdots & B_{N-1,0,5(N-1)} & \hdots & B_{N-1,N-1,5(N-1)} \\
\end{pmatrix} .
\]
Then we rewrite \eqref{2d_full_final} as
\begin{equation*}
 \b C (\vartheta ) \bm\zeta = \b m (\vartheta), 
\end{equation*}
and we consider $K$ detection directions, meaning $\vartheta^{(k)} , \, k = 1,...,K ,$ such that the system 
\begin{equation}\label{eq_final_2d}
\b D \bm\zeta = \b d ,
\end{equation}
for 
\begin{equation*}
\b D = \begin{pmatrix}
\b C (\vartheta^{(1)} ) \\ 
\vdots \\
\b C (\vartheta^{(K)} )
\end{pmatrix} \in \C^{K(5N-4)\times N^2}, \quad \text{and} \quad 
\b d = \begin{pmatrix}
\b m (\vartheta^{(1)} ) \\ 
\vdots \\
\b m (\vartheta^{(K)} )
\end{pmatrix} \in \C^{K(5N-4)},
\end{equation*}
is at least exactly determined.

\section{Numerical Results}\label{section_results}

Both linear systems derived in the previous section admit the general form
\begin{equation}\label{equation_final}
\b G \b x = \b g .
\end{equation}
In the case of depth-dependent coefficients, see \autoref{section_depth}, we have
\[
\b G := \b A \in \C^{(3N-2)\times N}, \quad \b x := \bm \gamma \in \R^N , \quad \b g := \b m \in \C^{3N-2},
\]
and in the case of constant in one direction coefficients, see \autoref{section_2d}, we get
\[
\b G := \b D \in \C^{K(5N-4)\times N^2}, \quad \b x := \bm \zeta \in \R^{N^2}, \quad \b g := \b d \in \C^{K(5N-4)}.
\]
We approximate the solution of \eqref{equation_final} by minimizing the Tikhonov functional
\[ 
\| \b G \b x - \b g\|_2^2 + \lambda \| \b x \|_2^2 ,
\]
where $\lambda >0$ is the regularization parameter.  Since $\b x$ is in both case a real-valued function we actually solve the following regularized equation
\[
\left( \Re (\b G)^\top  \Re (\b G) + \Im (\b G)^\top  \Im (\b G)+ \lambda \b I \right) \b x =  \Re (\b G)^\top  \Re (\b g) + \Im (\b G)^\top  \Im (\b g) ,
\]
where $\b I$ is the identity matrix with dimensions depending on each problem. We consider also noisy data for both measurements data, the pressure $p$ and the OCT data $m,$ with respect to the $L^2$ norm
\[
p_\delta = p + \delta_p \frac{\|  p \|_2}{\|  v \|_2} v , \quad 
\text{and} \quad m_\delta = m + \delta_m \frac{\|  m \|_2}{\|  w \|_2} w ,
\]
for given noise levels $\delta_p , \delta_m$ and $v = v_1 + \i v_2, \, w = w_1 + \i w_2,$ for $v_1 , v_2 , w_1$ and $w_2$ normally identically distributed, independent random variables. 

We present reconstructions for different functions $\gamma$ (related to $1/\Gamma$) and $\psi$ (related to $\hat\chi$) for both cases of media. As OCT data we consider the function $\tilde \chi$ (using the Fourier transform and the Kramers--Kronig relation) and to construct the simulated PAT data we have to assume that both functions have similar behavior such that ratio $\hat\chi / \gamma$ (see \eqref{eqPATOCTmeasSimplified}) is still integrable. In all figures we plot the spatial domain $\Omega_L.$

\begin{figure}[t]
\begin{center}
\includegraphics[scale=0.7]{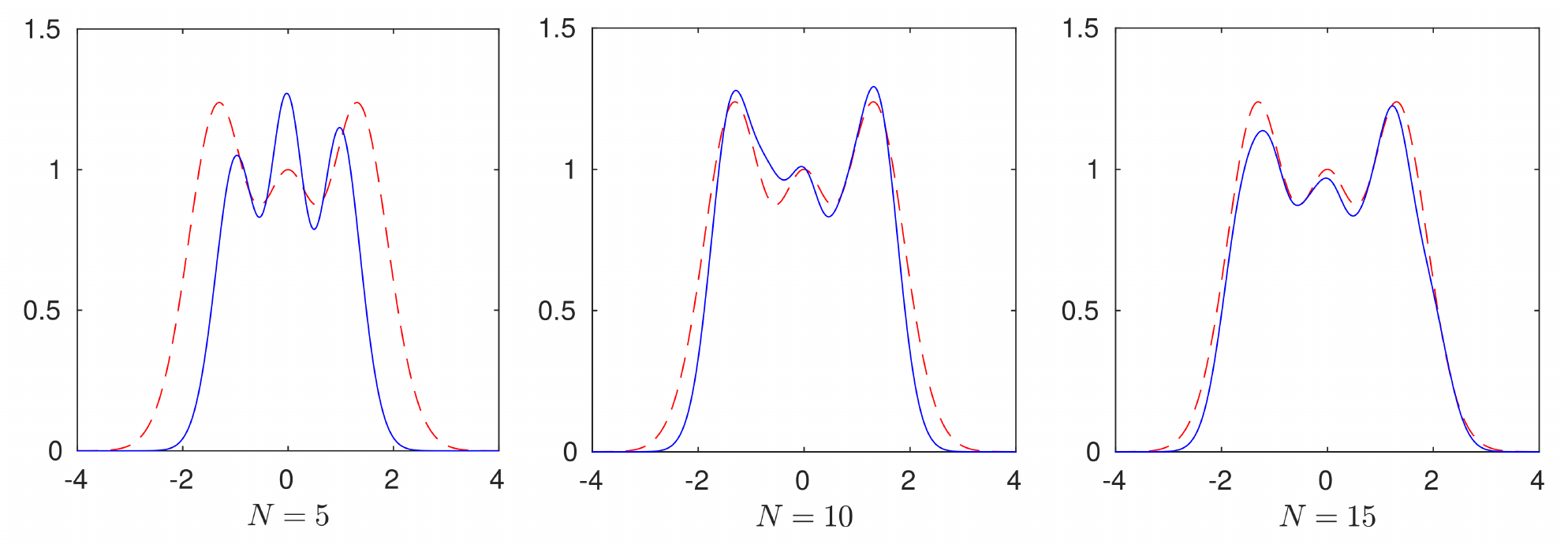}
\caption{Reconstruction of $\gamma,$ given by equation \eqref{gamma_example1}, for increasing number of Fourier coefficients.   }\label{Fig1}
\end{center}
\end{figure}

\subsection{Examples with depth-dependent coefficients (see \autoref{section_depth})}\label{section_depth_numerics}  
In the following figures the true curve is represented by a dashed red line and the reconstructed by a solid blue line.  Let $x\in\R.$ In the first example we consider 
\begin{equation}\label{gamma_example1}
\gamma (x) = (2x^4 + 1) \,\e^{-x^2},
\end{equation}
and
\[
\Im (\psi (\omega,x)) = h_1 (\omega) (x^4 + x^3 + x^2 + 0.1) \, \e^{-2x^2} .
\]
We set $\Omega = [-3.5,3.5]$ and $\Omega_L = [-4,4]$ such that 
$\supp \psi (\omega,\cdot ) \subset \Omega,$ and 
$\supp \gamma \subset \Omega_L,$  and we restrict ourselves to $\omega \in \mathcal{W}:= [-4,4].$ We consider data with $\delta_p = \delta_m = 3\%$ noise. The results are presented in \autoref{Fig1} for regularization parameter $\lambda = 10^{-4}$ and different values of $N.$ Here, we see the improvement in the reconstructions as $N$ increases.

\begin{figure}[t]
\begin{center}
\includegraphics[scale=0.7]{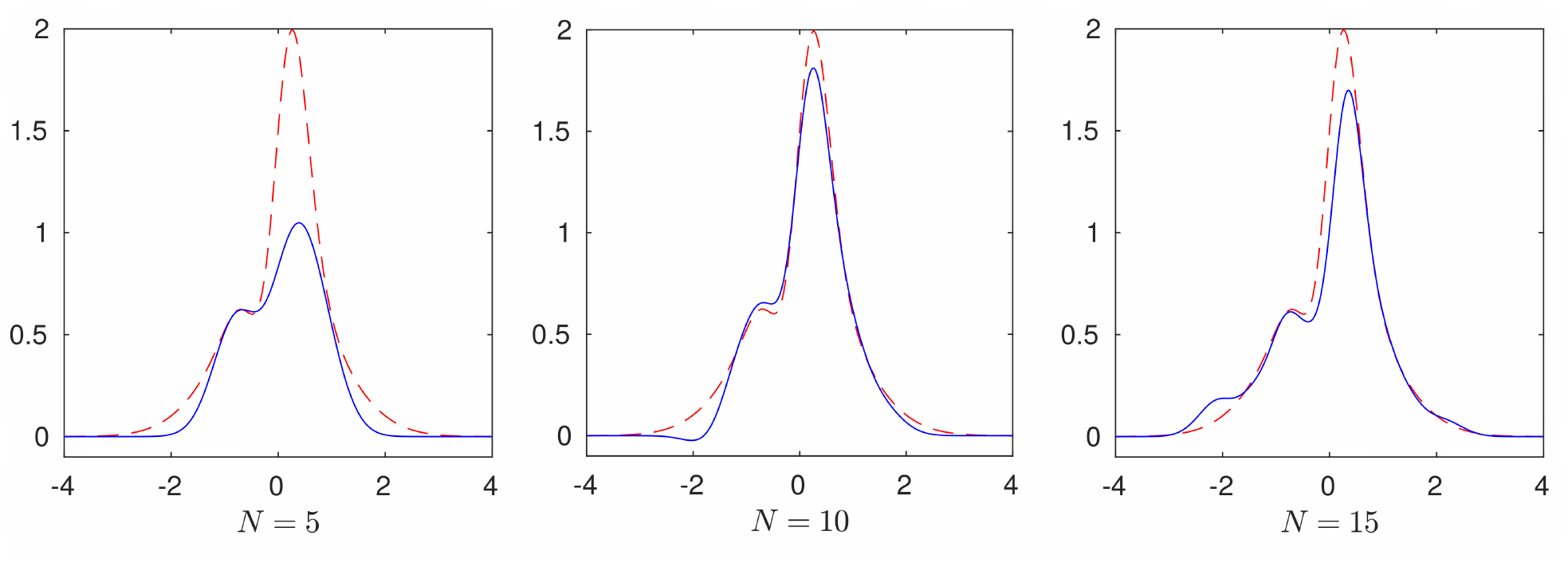}
\caption{Reconstruction of $\gamma,$ given by equation \eqref{gamma_example2}, for increasing number of Fourier coefficients.   }\label{Fig2}
\end{center}
\end{figure}

In the second example, we use the same function $\psi$ and we consider as $\gamma$ the function
\begin{equation}\label{gamma_example2}
\gamma (x) = h_0 (x) + h_0 (2x) + h_1 (3x) .
\end{equation}
We keep all the parameters the same as in the first example. In \autoref{Fig2}, we see the reconstructions for different number of coefficients. 

In the third example, we consider 
\begin{equation}\label{susce_example3}
\Im (\psi (\omega,x)) = (h_1 (\omega)+h_1 (2\omega) )(x^2 + 0.1) \, \e^{-2x^2} ,
\end{equation}
such that again $\supp \psi (\omega,\cdot ) \subset [-3.5,3.5],$ see the left picture in \autoref{Fig3}. We present the reconstructions of $\Im(\psi)$ using the form \eqref{gamma_example1} for $\gamma,$  while keeping all the other parameters the same. We set $N=15$ coefficients. We present the results for $\Im (\psi (\omega,1)), \, \omega \in \mathcal{W},$ (center picture) and $\Im (\psi (3,x)), \, x\in \Omega ,$
(right picture) in  \autoref{Fig3}.

\begin{figure}[t]
\begin{center}
\includegraphics[scale=0.7]{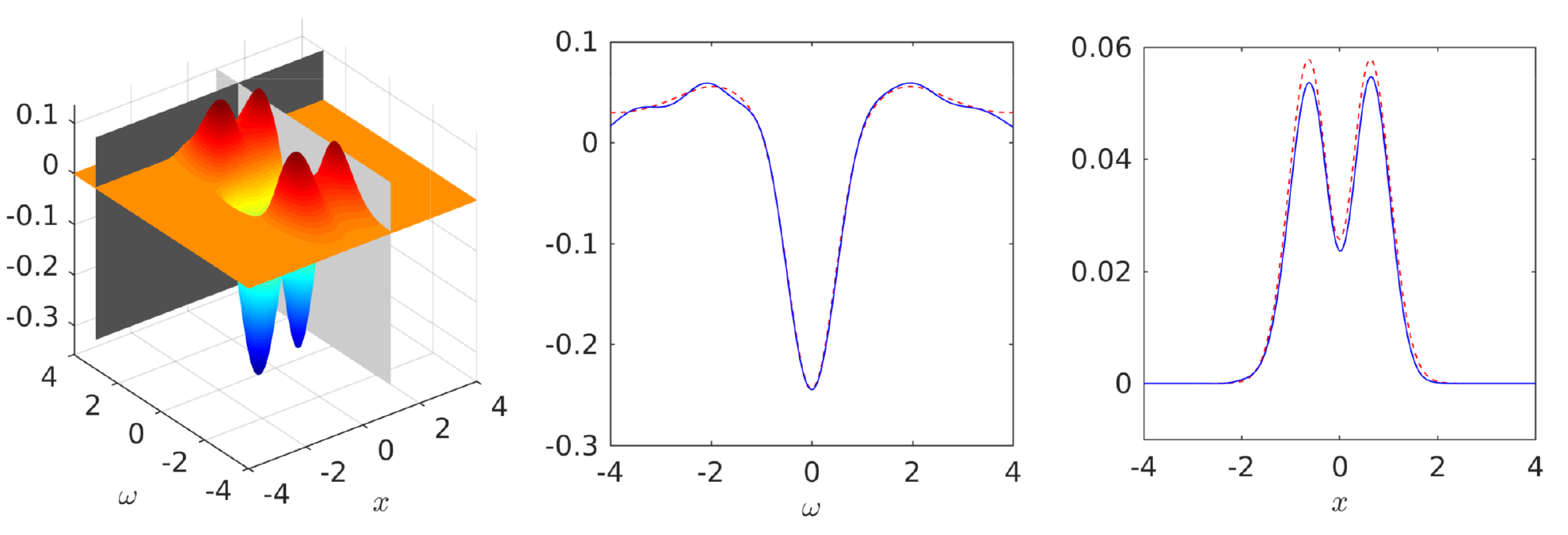}
\caption{Reconstruction of $\Im (\psi ),$ see equation \eqref{susce_example3}, for $N=15$. The true imaginary part (left), the cross section of the reconstruction at the line $x=1$ (center) and at the line $\omega=3$ (right).}\label{Fig3}
\end{center}
\end{figure}

\subsection{Examples with coefficients constant in one direction (see \autoref{section_2d})}

Here, the measurements are given at points on a line. We consider the minimum amount of measurement points in order to have an exactly determined system (\ref{eq_final_2d}) in our examples. 
In the following examples we keep the same noise levels $\delta_p = \delta_m = 3\%$ and we obtain the regularization parameter using the L-curve criterion \cite{HanOle93}. 

Let $x,y  \in \R .$  In the fourth example, we consider
\begin{equation}\label{gamma_example4}
\gamma (x,y) = \e^{ -(x+1.5)^2- (y+1.5)^2},
\end{equation}
and
\begin{equation}\label{susce_example4}
\Im (\psi (\omega,x,y)) = 0.7 (h_1 (\omega)+h_1 (2\omega)) \, \e^{ -(x+1.6)^4- \tfrac12 (y+1.6)^4}.
\end{equation}

We set $\Omega = [-4,4]^2,  \, \Omega_L = [-4.5,4.5]^2$  and $\mathcal{W}=[-3,3].$ The reconstructions of $\gamma$ for $N=5$ and $\vartheta^{(1)}=(0, \, 0,\, 1)^\top$ are presented in \autoref{Fig4}. The results for the cross-section of the imaginary part of $\psi$, given by equation \eqref{susce_example4}, at frequency $\omega=0$ are presented in \autoref{Fig5}. 
 
 \begin{figure}[t]
\begin{center}
\includegraphics[scale=0.7]{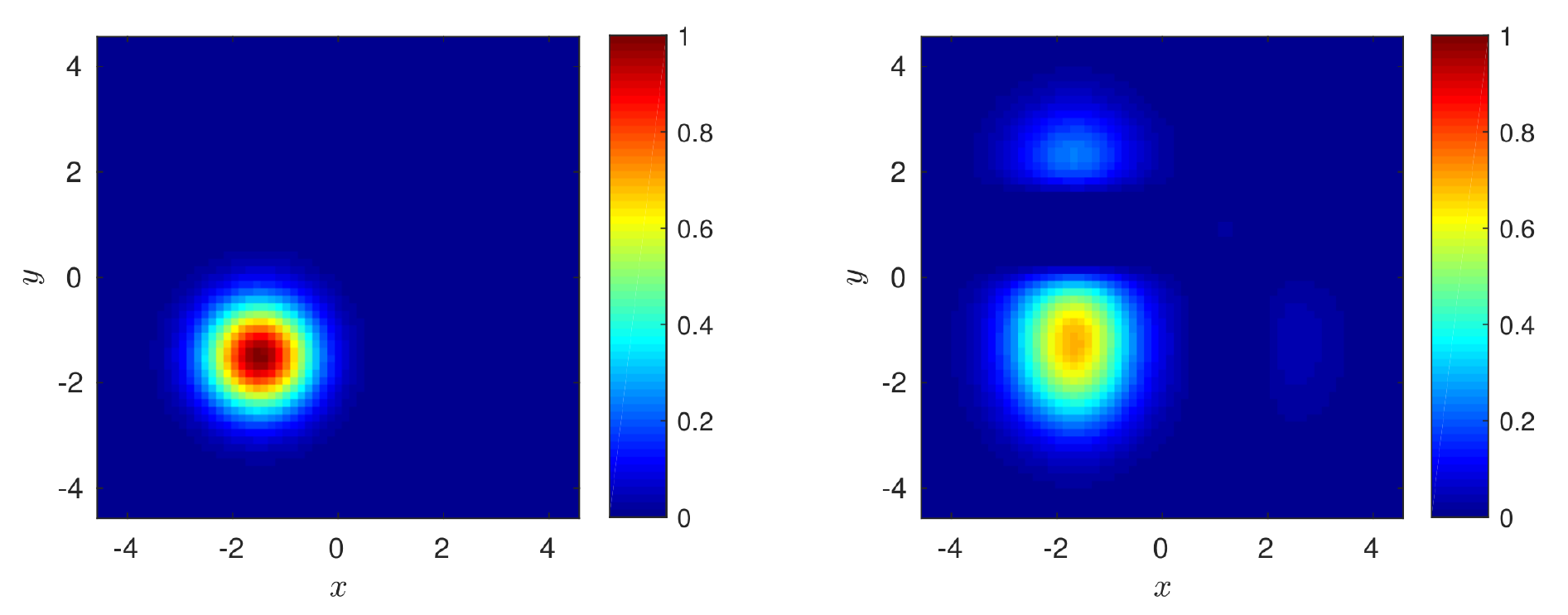}
\caption{The exact function $\gamma,$ see equation \eqref{gamma_example4}, (left) and  the reconstructed for $N=5$ and one measurement point (right).}\label{Fig4}
\end{center}
\end{figure} 
 
In the last example the unknown function is given by
\begin{equation}\label{gamma_example5}
\gamma (x,y) =  \e^{ -(x+0.5)^2- (y+2)^2}  + 0.8\,\e^{ -(x+2)^2- (y+0.5)^2}  +  \e^{ -(x-2)^2- (y-2)^2} .
\end{equation}
 The size of the medium is kept the same as in the previous example and we set $\mathcal{W}=[-2,2]$. Here, we want to test the performance of our numerical scheme with respect to the number of the detection directions $\vartheta^{(k)}.$ In the case of three measurement directions, see \eqref{eq_final_2d}, we consider
\[
\vartheta^{(1)} = \left(0, \, \cos (\tfrac{5\pi}{12} ),\, \sin (\tfrac{5\pi}{12}) \right)^\top , \quad \vartheta^{(2)} = \left(0, \,0,\,1\right)^\top , \quad \vartheta^{(3)} = \left(0, \, -\cos (\tfrac{5\pi}{12}),\, \sin (\tfrac{5\pi}{12}) \right)^\top .
\]
The reconstructions for $N=5$ coefficients are presented in \autoref{Fig6}, where  we set to zero the negative values.
  We set the imaginary part of $\psi$ to be
\begin{equation}\label{susce_example5}
\Im (\psi (\omega,x,y)) = h_1 (\omega) \left(\e^{ -(x+0.5)^4-  (y+2)^4} + \e^{ -0.6(x+2)^4-  (y+0.5)^4}  +  0.8\, \e^{ -(x-2)^4-  (y-2)^4} \right).
\end{equation}
The reconstruction for $\mathcal{W}=[-3,3]$ are given in \autoref{Fig7}, where we see the improvement of the results with respect to the Fourier coefficients. In the first case we set $N=5$ and we consider one detection direction. In the second case, we use $N=10$ coefficients and two measurement points in the directions:
\[
\vartheta^{(1)} = \left(0, \, \cos (\tfrac{7\pi}{16}),\, \sin (\tfrac{7\pi}{16}) \right)^\top , \quad \vartheta^{(2)} = \left(0, \, -\cos (\tfrac{7\pi}{16}),\, \sin (\tfrac{7\pi}{16}) \right)^\top .
\]

\begin{figure}[t]
\begin{center}
\includegraphics[scale=0.7]{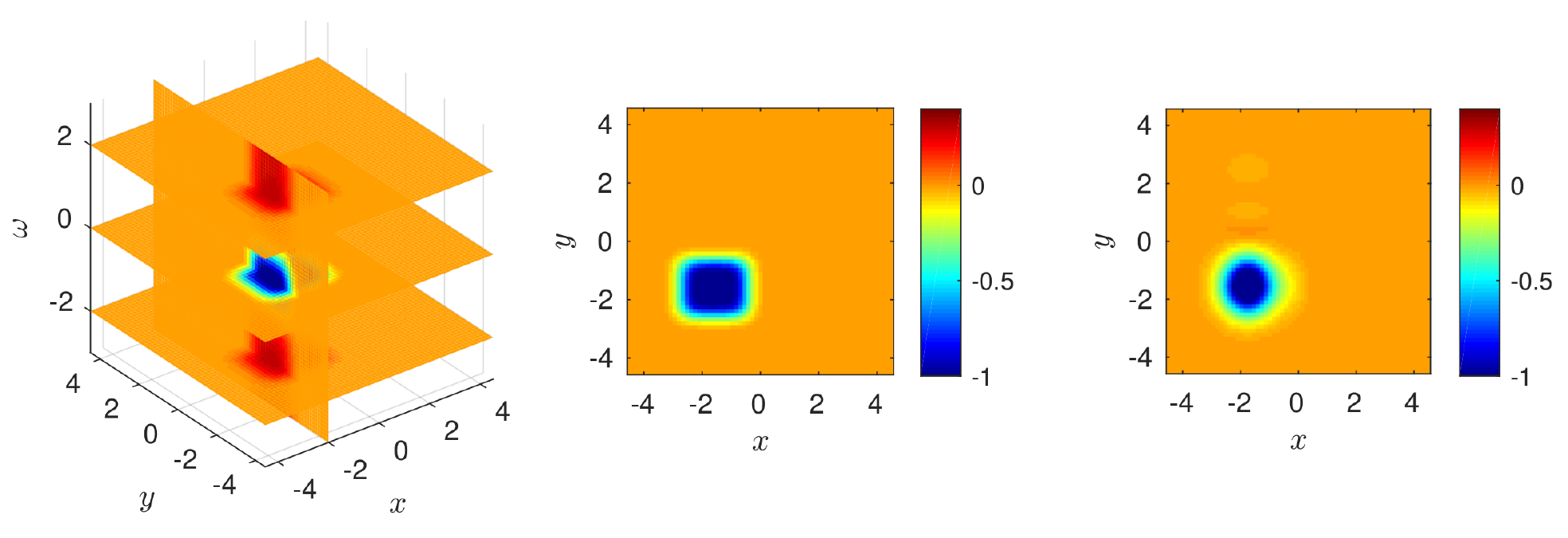}
\caption{Reconstruction of $\Im (\psi ),$ given by equation \eqref{susce_example4}, for $N=5$. The true imaginary part (left), the cross section of the true value at the plane $\omega=0$ (center) and the reconstructed (right).}\label{Fig5}
\end{center}
\end{figure}
 
\begin{figure}[t]
\begin{center}
\includegraphics[scale=0.7]{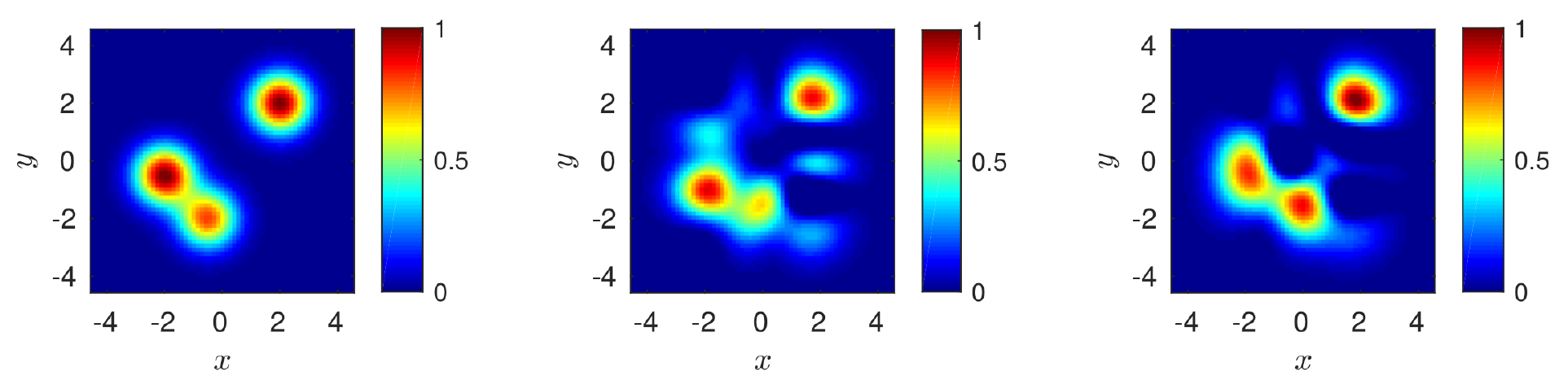}
\caption{The exact function $\gamma,$ see equation \eqref{gamma_example5}, (left) and the reconstructed for $N=5$ and one measurement point $K=1$ (center) and $K=3$ (right).}\label{Fig6}
\end{center}
\end{figure}

\begin{figure}[t]
\begin{center}
\includegraphics[scale=0.7]{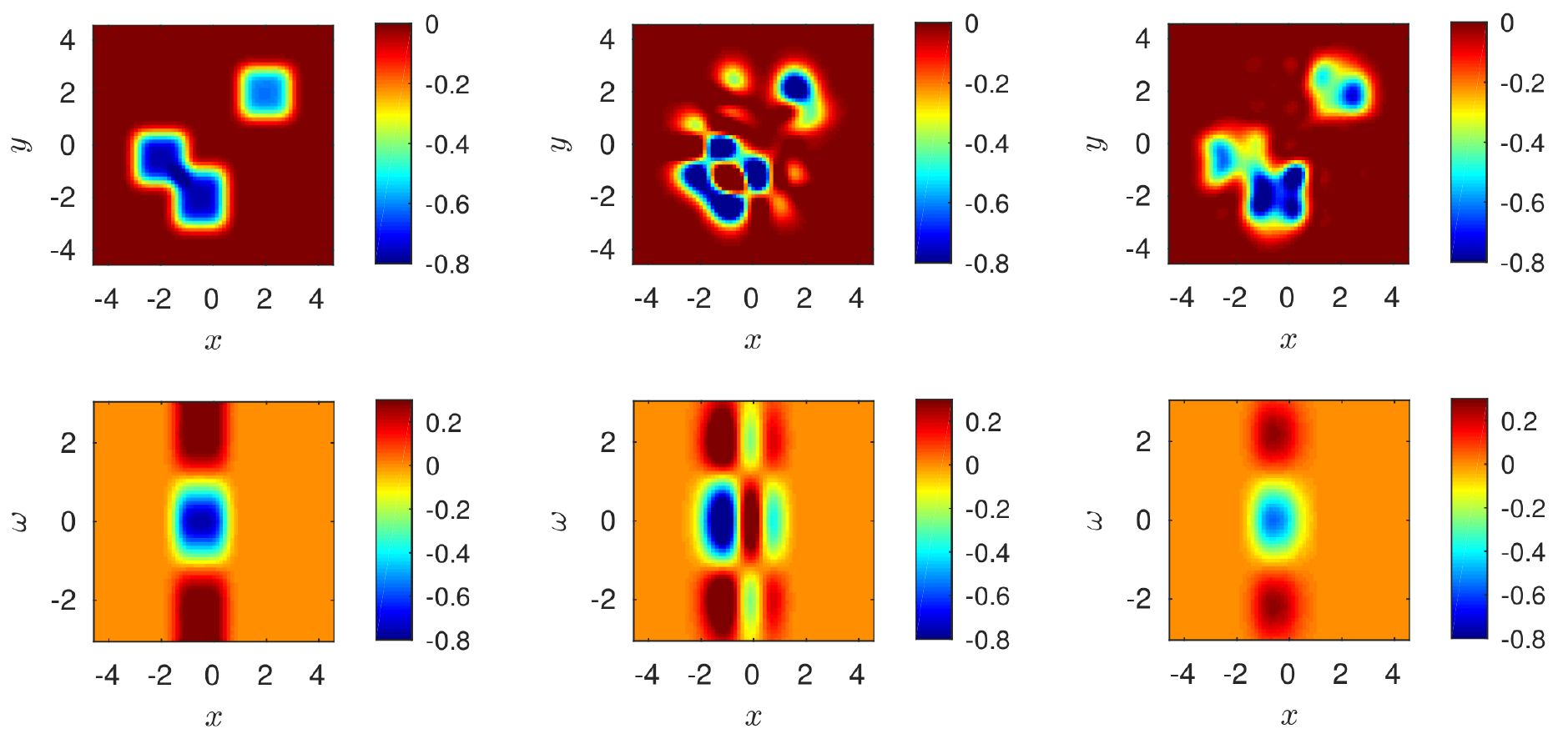}
\caption{Reconstruction of $\Im (\psi ),$ see equation \eqref{susce_example5}. In the left column we see the cross-section of the true imaginary part at the plane $\omega = 0$ (first row) and at the plane $y=-2$ (second row). The reconstructions for $N=5$ and one detection direction are presented in the second column. The results in the third column are for $N=10$ and two detection directions. }\label{Fig7}
\end{center}
\end{figure}

\section{Conclusions}
In this work we considered the inverse problem to reconstruct quantitatively the electric susceptibility and the Gr\"uneisen parameter of a non-magnetic linear dielectric medium from measurements with the multi-modal tomographic system of Photoacoustic and Optical Coherence Tomography. 
Our scheme is based on the numerical solution of a Fredholm integral equation of the first kind for the Gr\"uneisen parameter using a Galerkin type method. We presented numerical results for different kinds of media.

\section*{Acknowledgements}
The work of OS has been supported by the Austrian Science Fund (FWF), Project P26687-N25 (Interdisciplinary Coupled Physics Imaging).

\section*{Appendix}\label{section_her}

We recall Hermite functions and we present their properties which are used in this work. We connect the Fourier and Hilbert transforms of a function with expansions in terms of Hermite functions. 

Let $x\in \R.$ The normalized Hermite functions $h_k , \, k \in \N_0  $ are eigenfunctions of the inverse Fourier transform
\[
\check f (x) = \frac1{2\pi}\int_\R \hat f(\omega) \e^{-\i\omega t} \d \omega ,
\]
meaning they satisfy
\begin{equation*}
\check{h}_k (x) = (-\i)^k h_k (x) .
\end{equation*}

The product of two Hermite polynomials admits the following series expansion 
\begin{equation*}
H_k (x) H_l (x) = k! l! \sum_{m=0}^{\min (k,l)} 
\frac{2^m}{m! (k-m)! (l-m)!} H_{k+l-2m} (x) ,
\end{equation*}
also known as Feldheim's identity. Using \eqref{her_fun} we see that the product of two Hermite functions can be written as
\begin{equation}\label{her_pro_fun}
h_k (x) h_l (x) = \e^{-\tfrac{x^2}{2}} \sum_{m=0}^{\min (k,l)} 
\beta_{k,l,m} h_{k+l-2m} (x) ,
\end{equation}
for
\[
\beta_{k,l,m} = \pi^{-\tfrac14}\frac{(k! l! (k+l-2m)!)^{\tfrac12}}{m! (k-m)! (l-m)!}
\]
We recall the addition formula \cite{Fel43, MagObeSon66}
\begin{equation}\label{her_sum}
H_k (x+y) = \sum_{m=0}^k \frac{k!}{(k-m)!m!} (2y)^{k-m} H_m(x) ,
\end{equation} 
the multiplication formula
\begin{equation*}
H_k (\rho x) = k!\sum_{m=0}^{\left[ \tfrac{k}{2}\right] } \frac{\rho^k}{m! (k-2m)!} \left( 1-\frac{1}{\rho^2}\right)^m H_{k-2m} (x),
\end{equation*} 

and the inverse explicit expression
\begin{equation}\label{her_inv}
x^k = \frac{k!}{2^k}\sum_{m=0}^{\left[ \tfrac{k}{2}\right] } \frac{1}{m!(k-2m)!} H_{k-2m} (x) .
\end{equation}

Let $f \in L^2 (\R).$ We consider the expansion 
\begin{equation*}
f(x) = \sum_{k=0}^\infty f_k h_k (x) ,
\end{equation*}
where the coefficients $f_k$ are defined by
\begin{equation*}
f_k = \int_\R f(x) h_k (x) \d x .
\end{equation*}
The Hilbert transform of $f$ admits the expansion 
\begin{equation}\label{her_hil}
\mathcal{H} [f] (x) = \sum_{k=0}^\infty \tilde f_k h_k (x) ,
\end{equation}
where $\tilde f_k$ are given by \cite{PorSchuKin13}
\begin{equation}\label{her_coe_hil}
\tilde f_k  = (-\i)^{k+1} \sum_{m=0}^\infty f_m (-\i)^m \int_\R \sign (x) h_k (x) h_m (x) \d x .
\end{equation}

For $\b k \in \N_0^d$ and $ x \in \R^d,$ we define the $\b k$th Hermite polynomial as
\begin{equation}\label{her_3d}
H_{\b k} ( x) = \prod_{j=1}^{d} H_{k_j} (x_j) . 
\end{equation}

Now we present the proof of \autoref{lemma1}.
\begin{proof}[\autoref{lemma1}]
We consider the convolution theorem for the inverse Fourier transform and the above properties.
\begin{equation*}
\begin{aligned}
\int_\R \e^{-\tfrac{x^2}2} h_k (x) \e^{- \i \omega  x } \d x &= \left( \int_\R \e^{-\tfrac{x^2}2}  \e^{- \i \omega  x } \d x\right) \ast \check{h}_k (\omega) \\
&= 2\pi (-\i)^k \e^{-
\tfrac{\omega^2}2} \ast h_k (\omega) \\
&= 2\pi (-\i)^k \int_\R \e^{-
\tfrac{(\omega-y)^2}2}  h_k (y) \d y \\
&= 2\pi (-\i)^k \alpha_k \e^{-
\tfrac{\omega^2}4} \int_\R \e^{-(y-
\tfrac{\omega}2 )^2}  H_k (y) \d y \\
&= 2\pi (-\i)^k \alpha_k \e^{-
\tfrac{\omega^2}4} \int_\R \e^{-z^2}  H_k (z + \tfrac{\omega}2) \d z .
\end{aligned}
\end{equation*}
To compute the last integral we apply the formula \eqref{her_sum}
\begin{equation*}
\begin{aligned}
\int_\R \e^{-z^2}  H_k (z + \tfrac{\omega}2) \d z &= \sum_{m=0}^k \frac{k!}{(k-m)!m!} \omega^{k-m}  \int_\R \e^{-z^2} H_m(z) \d z \\ 
&= \sum_{m=0}^k \frac{k!}{(k-m)!m!} \omega^{k-m}  \int_\R \e^{-z^2} H_m(z) H_0 (z) \d z \\
&= \sum_{m=0}^k \frac{k!}{(k-m)!m!} \omega^{k-m} a_m^{-2} \delta_{m,0 } \\
&= \sqrt{\pi} \omega^k .
\end{aligned}
\end{equation*}
The last two equations result to \eqref{eq_lemma1}.
\end{proof}

\section*{References}
\printbibliography[heading=none]

\end{document}